\documentclass[11pt,a4paper]{article}

\usepackage{epsf,epsfig,amsfonts,amsgen,amsmath,amstext,amsbsy,amsopn,amsthm,cases,listings,color
}
\usepackage{ebezier,eepic}
\usepackage{color}
\usepackage{multirow}
\usepackage{epstopdf}
\usepackage{graphicx}
\usepackage{pgf,tikz}
\usepackage{mathrsfs}
\usepackage[marginal]{footmisc}
\usepackage{enumitem}
\usepackage[titletoc]{appendix}
\usepackage{booktabs}
\usepackage{url}
\usepackage{mathtools}

\usepackage{pgfplots}
\usepackage{authblk}
\usepackage{amssymb}

\usepackage{wasysym}

\usepackage{empheq}

\usepackage{dsfont}

\usepackage{subfigure}
\pgfplotsset{compat=1.18}
\usepackage{mathrsfs}
\usepackage{wasysym} 
\usetikzlibrary{arrows}
\usepackage{aligned-overset}
\usepackage{bm}
\usepackage[backref=page]{hyperref}

\allowdisplaybreaks[1]

\definecolor{uuuuuu}{rgb}{0.27,0.27,0.27}
\definecolor{sqsqsq}{rgb}{0.1255,0.1255,0.1255}

\newtheorem{definition}{Definition} [section]
\newtheorem{theorem}[definition]{Theorem}
\newtheorem{lemma}[definition]{Lemma}
\newtheorem{proposition}[definition]{Proposition}

\newtheorem{claim}[definition]{Claim}

\newtheorem{fact}[definition]{Fact}
\newenvironment{pf}{\noindent{\bf Proof}}

\begin{document}
\title{\bf\Large Many vertex-disjoint even cycles of fixed length in a graph}
\date{\today}
\author[1]{Jianfeng Hou\thanks{Research was supported by National Natural Science Foundation of China (Grant No. 12071077). Email: \texttt{jfhou@fzu.edu.cn}}}
\author[1]{Caiyun Hu\thanks{Email: \texttt{hucaiyun.fzu@gmail.com}}}
\author[2]{Heng Li\thanks{Email: \texttt{heng.li@sdu.edu.cn}}}
\author[3]{Xizhi Liu\thanks{Research was supported by ERC Advanced Grant 101020255 and Leverhulme Research Project Grant RPG-2018-424. Email: \texttt{xizhi.liu.ac@gmail.com}}}
\author[1]{Caihong Yang\thanks{Email: \texttt{chyang.fzu@gmail.com}}}
\author[1]{Yixiao Zhang\thanks{Email: \texttt{fzuzyx@gmail.com}}}
\affil[1]{Center for Discrete Mathematics,
            Fuzhou University, Fujian, 350003, China}
\affil[2]{School of Mathematics, Shandong University, 
            Shandong, 250100, China}
\affil[3]{Mathematics Institute and DIMAP,
            University of Warwick, 
            Coventry, CV4 7AL, UK}
\maketitle
\begin{abstract}
For every integer $k \ge 3$, 
we determine the extremal structure of an $n$-vertex graph with at most $t$ vertex-disjoint copies of $C_{2k}$ when $n$ is sufficiently large and $t$ lies in the interval $\left[\frac{\mathrm{ex}(n,C_{2k})}{\varepsilon n}, \varepsilon n\right]$, where $\varepsilon>0$ is a constant depending only on $k$. 
The question for $k = 2$ and $t = o\left(\frac{\mathrm{ex}(n,C_{2k})}{n}\right)$ was explored in prior work~\cite{HHLLYZ23a}, 
revealing different extremal structures in these cases. 
Our result can be viewed as an extension of the theorems by Egawa~\cite{Ega96} and Verstra\"{e}te~\cite{Ver03}, where the focus was on the existence of many vertex-disjoint cycles of the same length without any length constraints.

\medskip

\end{abstract}
\section{Intorduction}\label{SEC:Intorduction}
The classical Corr\'{a}di--Hajnal Theorem~\cite{CH63} asserts that for every $t \ge 2$, a graph with minimum degree at least $2t$ on at least $3t$ vertices contains $t$ pairwise vertex-disjoint cycles. 
Thomassen~\cite{Tho83} extends this to graphs with minimum degree at least $3t+1$ on at least $t^{2+o(1)}$ vertices, showing they contain $t$ pairwise vertex-disjoint cycles of the same length.
Thomassen's result was further extended by
Egawa~\cite{Ega96} to graphs with minimum degree at least $2t$ on at least $(17+o(1))t$ vertices when $t \ge 3$. 
In~\cite{Ver03}, Verstra\"{e}te provided a concise proof, showing that for every $t \ge 2$, a graph with minimum degree at least $2t$ on at least $t^{7t}$ vertices also contains $t$ pairwise vertex-disjoint cycles of the same length.
In particular, the case $t=2$ confirms a conjecture of H\"{a}ggkvist~\cite{Tho83}.

There are some other extensions of the Corr\'{a}di--Hajnal Theorem from the average degree perspective or by imposing different constraints on the length of cycles. 
Notable instances include the work of Justesen~\cite{Jus85}, who demonstrated that graphs with $n \ge 3t$ vertices and at least $\max\left\{(2t-1)(n-t), \binom{3t-1}{2}+(n-3t+1)\right\}$ edges contains $t$ pairwise vertex-disjoint cycles.
Chiba et al.~\cite{CFKS14} showed that every graph with minimum degree  at least $2t$ and with sufficiently large number of vertices contains $t$ pairwise vertex-disjoint even cycles except in a few exceptional cases. 
Wang~\cite{Wang12} proved that for $t \ge 2$, every graph with minimum degree at least $2t$ and at least $4t$ vertices contains $t$ pairwise vertex-disjoint cycles of length at least four, except in three exceptional cases. 
Harvey and Wood~\cite{HW15} extended this to $t \ge 6$ and $\ell \ge 3$, demonstrating that every graph with average degree at least $4\ell t/3$ contains $t$ pairwise vertex-disjoint cycles of length at least $r$. 

The problem of finding cycles of a fixed length in a graph also attracted considerable attention and stands as a central topic in Extremal Graph Theory. 
An old theorem due to Mantel~\cite{Mantel07} states that every $n$-vertex graph with at least $\lfloor n^2/4 \rfloor + 1$ edges contains a triangle. 
For every $k \ge 2$, using the Stability method, Simonovits~\cite{Sim66} showed that every $n$-vertex graph with at least $\lfloor n^2/4 \rfloor + 1$ edges contains a copy of $C_{2k+1}$ when $n$ is sufficiently large.
The tightness of the bound $\lfloor n^2/4 \rfloor + 1$ is demonstrated by the balanced complete bipartite graph on $n$ vertices in both cases.
Finding even cycles of fixed length is much more challenging. 
Erd{\H o}s stated (without a proof) in~\cite{Erd64} that for every $k \ge 2$, every $n$-vertex graph with at least $A_k n^{1+\frac{1}{k}}$ edges contains a copy of $C_{2k}$, where $A_k$ is a constant depending on $k$. 
The classical Bondy-Simonovits Theorem~\cite{BS74} strengthens this by showing the existence of $C_{2\ell}$ in an $n$-vertex graph with at least $100k n^{1+\frac{1}{k}}$ edges for all integers $\ell \in [k, kn^{1/k}]$.  
Subsequent works~\cite{Ver00,Oleg12,BJ17,BJ17b,He21} have refined this bound. 
Determining whether the exponent $1+ 1/k$ is indeed tight remains a major open problem in Extremal Graph Theory. 
Currently, tightness is only known for $k\in \{2,3,5\}$, and we refer the reader to~\cite{ER62,ERS66,Brown66,Ben66,Fur83,Wen91,Fur94,LUW94,LUW95,LU95,Fur96,LUW99,Mel04,MM05,FNV06} for more related results.  

In this note, we explore the following problem, which can be seen as a collective extension of the results by Egawa, Verstra\"{e}te, and Bondy--Simonovits: 
\begin{center}
    What is the maximum number of edges in an $n$-vertex graph without $t+1$ pairwise vertex-disjoint cycles of a fixed length?
\end{center}
This study is motivated by a remarkable result of Allen et al.~\cite{ABHP15}, where they determined, for large $n$, both the maximum number of edges  and the structures of four distinct extremal constructions (corresponding to different values of $t$) in an $n$-vertex graph without $t+1$ pairwise vertex-disjoint triangles for all integers $t$ in the interval $[0, n/3]$. 
Their result builds upon old theorems established by of Erd{\H o}s~\cite{Erdos62}, Moon~\cite{Moon68}, and Simonovits~\cite{SI68}, which specifically address this problem in cases where $t$ is relatively small compared to $n$. 

Let us introduce some definitions before presenting the main result. 
Given a graph $F$, let $\mathrm{ex}(n,F)$ denote the maximum number of edges in an $n$-vertex $F$-free graph. 
We expand this notation to $\mathrm{ex}(n,(t+1)F)$, representing the maximum number of edges in an $n$-vertex graph containing at most $t$ pairwise vertex-disjoint copies of $F$.

The results of Mantel and Simonovits can be expressed as $\mathrm{ex}(n,C_{2k+1}) = \lfloor n^2/4 \rfloor$ for $k \ge 1$ when $n$ is sufficiently large. 
Similarly, the results of Erd{\H o}s and Bondy--Simonovits can be written as $\mathrm{ex}(n,C_{2k}) \le 100 k n^{1+\frac{1}{k}}$ for all $k \ge 2$. 
The exact value of $\mathrm{ex}(n,(t+1)C_{2k+1})$ was established in~\cite{HLLYZ23} for large $n$ and $t = o(n)$. 
Additionally, the exact value of $\mathrm{ex}(n,(t+1)C_{2k})$ was determined in~\cite{HHLLYZ23a} for large $n$, with $t$ satisfying $t = o\left(\frac{\mathrm{ex}(n,C_{2k})}{n}\right)$ or $\sqrt{n} \ll t \ll n$.
The following result builds upon the previous one, determining the exact value of $\mathrm{ex}(n,(t+1)C_{2k})$ for $\frac{\mathrm{ex}(n,C_{2k})}{n} \ll t \ll n$.

\begin{theorem}\label{THM:2nd-C2k-main}
    For every integer $k \ge 3$ there exists $\varepsilon_k > 0$ such that  for sufficiently large $n$, 
    \begin{align*}
        \mathrm{ex}(n,(t+1)C_{2k}) 
        = \binom{n}{2}-\binom{n-k(t+1)+1}{2} + 1
    \end{align*}
    holds for all integers $t \in \left[\frac{\mathrm{ex}(n,C_{2k})}{\varepsilon_k n}, \varepsilon_{k} n\right]$. 
\end{theorem}

Define $S(m,n)$ as the graph obtained from an $m$ by $n-m$ complete bipartite graph by adding all pairs in the part of size $m$ as edges. Let $S^{+}(m,n)$ be the graph obtained from $S(m,n)$ by adding an edge into the part of size $n-m$. 
Our proof shows that the unique extremal construction for Theorem~\ref{THM:2nd-C2k-main} is simply $S^{+}(k(t+1)-1,n)$. 
In contrast, the extreml construction for $\mathrm{ex}(n,(t+1)C_{2k})$ when $t = o\left(\frac{\mathrm{ex}(n,C_{2k})}{n}\right)$ is different, and refer the reader to~\cite{HHLLYZ23a} for more details.

Given a graph $F$ and an integer $q \le v(F)$, let 
\begin{align*}
    F[q]
    := \left\{F[S] \colon S\subset V(F) \text{ and } |S| \ge q \right\}.
\end{align*}
For an integer $q$ satisfying $v(F) - \alpha(F) \le q \le v(F)$,  
let 
\begin{align*}
    F^{\mathrm{ind}}[q]
    := \left\{F[S] \colon S\subset V(F),\ |S| \ge q, \text{ and } V(F)\setminus S \text{ is independent}  \right\}.
\end{align*}
In other words, $F[q]$ consists of induced subgraphs of $F$ that can be obtained by removing at most $v(F)-q$ vertices, while $F^{\mathrm{ind}}[q]$ consists of induced subgraphs of $F$ that can be obtained by removing an independent set of size at most $v(F)-q$. 

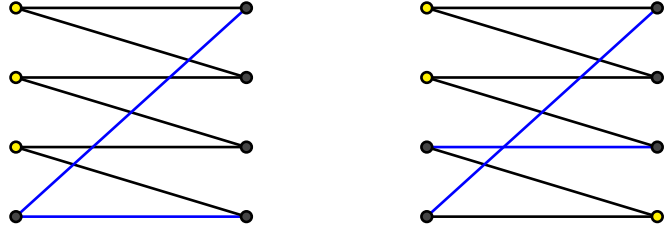
\begin{figure}[htbp]
\centering
\tikzset{every picture/.style={line width=1pt}} 
\begin{tikzpicture}[x=0.75pt,y=0.75pt,yscale=-1,xscale=1]

\draw    (175,70) -- (290,70) ;
\draw    (175,105) -- (290,105) ;
\draw    (175,140) -- (290,140) ;
\draw  [blue]  (175,175) -- (290,175) ;
\draw    (175,70) -- (290,105) ;
\draw    (175,105) -- (290,140) ;
\draw    (175,140) -- (290,175) ;
\draw  [blue]  (175,175) -- (290,70) ;
\draw [fill=yellow]   (175,70) circle (2pt);
\draw [fill=yellow]   (175,105) circle (2pt);
\draw [fill=yellow]   (175,140)  circle (2pt);
\draw [fill=uuuuuu]   (175,175)  circle (2pt);
\draw [fill=uuuuuu]   (290,70)  circle (2pt);
\draw [fill=uuuuuu]   (290,105) circle (2pt);
\draw [fill=uuuuuu]   (290,140)  circle (2pt);
\draw [fill=uuuuuu]   (290,175)  circle (2pt);
%
\draw    (380,70) -- (495,70) ;
\draw    (380,105) -- (495,105) ;
\draw  [blue]  (380,140) -- (495,140) ;
\draw    (380,175) -- (495,175) ;
\draw    (380,70) -- (495,105) ;
\draw    (380,105) -- (495,140) ;
\draw    (380,140) -- (495,175) ;
\draw  [blue]  (380,175) -- (495,70) ;
\draw [fill=yellow]   (380,70) circle (2pt);
\draw [fill=yellow]   (380,105) circle (2pt);
\draw [fill=uuuuuu]   (380,140)  circle (2pt);
\draw [fill=uuuuuu]   (380,175)  circle (2pt);
\draw [fill=uuuuuu]   (495,70)  circle (2pt);
\draw [fill=uuuuuu]   (495,105) circle (2pt);
\draw [fill=uuuuuu]   (495,140)  circle (2pt);
\draw [fill=yellow]   (495,175)  circle (2pt);
\end{tikzpicture}
\caption{Two examples of ways to remove an independent set of size three from $C_{8}$.}
\label{fig:C2k-[k+1]}
\end{figure}

Recall the definition of the covering number $\tau(F)$ of a graph $F$ is 
\begin{align*}
    \tau(F) 
    := \min \left\{|S| \colon \text{$S \subset V(F)$ and $S\cap e\neq \emptyset$ for all $e\in F$} \right\}.
\end{align*}
A minor adjustment to the proof of Theorem~\ref{THM:2nd-C2k-main} leads to the following more general result. For simplicity, we omit its details in this note.

\begin{theorem}\label{THM:2nd-general-biparite-main}
    Let $s_2 \ge s_1 \ge 2$ be integers and $\mathbb{B}:=B_{s_1, s_2}$ be an $s_1$ by $s_2$ connected bipartite graph with $\tau(B) = s_1$. 
    Suppose that some member in $\mathbb{B}^{\mathrm{ind}}[s_2+1]$ is the union of a matching with isolated vertices.  
    Then there exists $\varepsilon >0$ such that for sufficiently large $n$, 
    \begin{align*}
        \mathrm{ex}(n,(t+1)\mathbb{B}) 
        = \binom{n}{2}-\binom{n-\tau(t+1)+1}{2} 
            +\mathrm{ex}(n-\tau(t+1)+1, \mathbb{B}[s_2+1]) 
    \end{align*}
    holds for all integers $t \in \left[\frac{\mathrm{ex}(n, \mathbb{B})}{\varepsilon n}, \varepsilon n\right]$. 
\end{theorem}

The extremal construction for Theorem~\ref{THM:2nd-general-biparite-main} is the graph obtained from $S(\tau(t+1)-1,n)$ by adding a maximum $\mathbb{B}[s_2+1]$-free graph to the part of size $n-m$.

The proof of Theorem~\ref{THM:2nd-C2k-main} is presented in Section~\ref{SEC:proof-2nd-C2k}. 
In the next section, we present some definitions and preliminary results. 
\section{Preliminaries}\label{SEC:prelim}
Given a graph $G$ and a set $W \subset V(G)$, we use $G[W]$ to denote the \textbf{induced subgraph} of $G$ on $W$, 
and use $G-W$ to denote the induced subgraph of $G$ on $V(G)\setminus W$. 
Given two disjoint sets $S, T \subset V(G)$, we use $G[S,T]$ to denote the collection of edges in $G$ that have nonempty intersection with both $S$ and $T$.  
We use $\overline{G}$ to denote the \textbf{complement} of $G$. 

We will use the following lower bound for $\mathrm{ex}(n,C_{2k})$ which follows from a standard probabilistic deletion argument. 

\begin{proposition}\label{PROP:C2k-free-lower-bound}
    For every integer $k \ge 2$ it holds that $\mathrm{ex}(n,C_{2k}) = \Omega\left(n^{1+\frac{1}{2k-1}}\right)$. 
\end{proposition}

The following results can be found in~\cite{HHLLYZ23a}. 

\begin{fact}\label{FACT:binom-inequality-b}
    Suppose that $n, \ell, x \ge 0$ are integers satisfying $\ell + x \le n$. 
    Then 
    \begin{align*}
        \binom{n}{2} - \binom{n-\ell-x}{2} 
        \ge \binom{n}{2} - \binom{n-\ell}{2} + x (n-\ell-x). 
    \end{align*}
\end{fact}
%
%
\begin{proposition}[{\cite[Proposition~2.8]{HHLLYZ23a}}]\label{PROP:Turan-ratio}
    Let $F$ be a connected graph. 
    For all integers $n \ge m \ge 1$ we have 
    \begin{align*}
        \frac{\mathrm{ex}(n,F)}{n}
        \ge \left(1 - \frac{m}{n}\right)\frac{\mathrm{ex}(m,F)}{m}. 
    \end{align*}
\end{proposition}

The following simple but useful lemma gives a rough upper bound for the number of edges in an $(t+1)F$-free $r$-graph with bounded degree. 

\begin{lemma}[{\cite[Lemma~2.13]{HHLLYZ23a}}]\label{LEMMA:trivial-max-degree}
    Suppose that $F$ is a graph on $m$ vertices. 
    Then for every $t \ge 0$, every $n$-vertex $(t+1)F$-free graph $G$ satisfies 
    \begin{align*}
        |G|
        \le mt \cdot \Delta(G) + \mathrm{ex}(n-mt, F). 
    \end{align*}
\end{lemma}

Given a family $\mathcal{F}$ of graphs and an integer $t \ge 0$ let 
\begin{align*}
    (t+1) \mathcal{F}
    := \left\{F_1 \sqcup \cdots \sqcup F_{t+1} \colon F_i \in \mathcal{F} \text{ for all } i\in [t+1]\right\}. 
\end{align*}

\begin{proposition}[{\cite[Proposition~5.8]{HHLLYZ23a}}]\label{PROP:2nd-graph-one-star-is-better}
    Let $s_2 \ge s_1 \ge 2$ be integers and $\mathbb{B}:= B_{s_1, s_2}$ be an $s_1$ by $s_2$ bipartite graph. 
    Let $\mathbb{B}':= \mathbb{B}[s_2+1]$. 
    Suppose that $t \le \frac{n}{2e(s_2+1)}$. 
    Then 
    \begin{align*}
        \mathrm{ex}\left(n,(t+1)\mathbb{B}'\right)
        \le \mathrm{ex}\left(n,\mathbb{B}'\sqcup tK_{1,s_2}\right)
        \le \binom{n}{2}-\binom{n-t}{2}+ \mathrm{ex}\left(n-t, \mathbb{B}'\right). 
    \end{align*}
\end{proposition}

Suppose that $G = G[V_1,V_2]$ and $B_{s_1,s_2} = B[W_1, W_2]$ are bipartite graphs. 
An \textbf{ordered copy} of $B_{s_1, s_2}$ in $G$ is a copy of $B_{s_1, s_2}$ in $G$ with the condition that $W_1$ is contained in $V_1$ and $W_2$ is contained in $V_2$. 

The following three technical lemmas are useful for finding many vertex-disjoint copies of a bipartite graph. 

\begin{lemma}[{\cite[Lemma~5.3]{HHLLYZ23a}}]\label{LEMMA:2nd-perfect-Ks1s2-bipartite-a}
    Let $s_2 \ge s_1 \ge 1$ be integers and $\alpha \in (0,1)$ be a real number.  
    The following statement holds for sufficiently large $n$. 
    Suppose that $G$ is an $m$ by $n$ bipartite graph on $V_1$ and $V_2$ that satisfies 
    \begin{align*}
        m \le \frac{(1-\alpha s_1)s_1}{s_2}n \quad\text{and}\quad 
        d_{G}(v)
        \ge (1-\alpha)n \quad\text{for all $v\in V_1$}. 
    \end{align*}
     Then $G$ contains $\left \lfloor  m/s_1\right \rfloor$  pairwise vertex-disjoint ordered copies of $K_{s_1,s_2}$.
\end{lemma}
\begin{lemma}[{\cite[Lemmas~5.1, 5.2]{HHLLYZ23a}}]\label{LEMMA:2nd-perfect-Ks1s2-bipartite}
    Let $s_2 \ge s_1 \ge 1$ be integers and $\alpha \in (0,1)$ be a real number. 
    The following holds for sufficiently large $n$.  
   Suppose that $G$ is an $m$ by $n$ bipartite graph on $V_1$ and $V_2$ such that $m \le \frac{\alpha n}{8s_2}$ and $d_{G}(v) \geq \alpha n$ for every $v \in V_1$. 
   Then $G$ contains at least $\left \lfloor  \frac{m}{s_1} - \frac{4}{\alpha}\right \rfloor$  pairwise vertex-disjoint ordered copies of $K_{s_1, s_2}$.
   
    If, in addition, there are at least $\min\left\{\frac{5s_1(s_1-1)}{\alpha},\ \frac{s_1-1}{s_1}m\right\}$ vertices in $V_1$ with degree at least $\left(1-\frac{\alpha}{2s_1}\right) n$, then $G$ contains $\left \lfloor  m/s_1\right \rfloor$  pairwise vertex-disjoint ordered copies of  $K_{s_1, s_2}$.
\end{lemma}

\begin{lemma}[{\cite[Lemma~5.4]{HHLLYZ23a}}]\label{LEMMA:2nd-weak-bound-Ks1s2}
    Let $s_2 \ge s_1 \ge 1$ be integers and $\alpha \in (0,1)$ be a real number.  
    Let $\mathbb{B}:= B_{s_1,s_2}$ be a bipartite graph with part sizes $s_1, s_2$. 
    For every $\alpha > 0$ there exists  $N_0$ such that the following statement holds for all $n \ge N_0$ and $t \le \frac{\alpha n}{18 (s_1+s_2)^2}$. 
    Every $n$-vertex $(t+1)\mathbb{B}$-free graph $G$ satisfies  
    \begin{align*}
        |G|
        \le \left(s_1t + \frac{34(s_1+s_2)}{\alpha} \right) \Delta(G)
            + \frac{\alpha s_1}{4} t n 
            + \mathrm{ex}(n-st,\mathbb{B}). 
    \end{align*}
    If, in addition, $\Delta(G) \le (1-\alpha)n$ and $t \ge \frac{5}{\alpha s_1}\left(\frac{\mathrm{ex}(n,\mathbb{B})}{n-1} + \frac{36(s_1+s2)}{\alpha}\right)$, then
    \begin{align*}
        |G| 
        \le \binom{n}{2}-\binom{n-s_1(t-1)}{2}. 
    \end{align*}
\end{lemma}
\section{Proof of Theorem~\ref{THM:2nd-C2k-main}}\label{SEC:proof-2nd-C2k}
In this section, we prove Theorem~\ref{THM:2nd-C2k-main}. 
Our proof strategy begins by considering the collection $L \subset V(G)$ of all vertices in a maximum $n$-vertex $(t+1)C_{2k}$-free graph $G$ with degree $\Omega(n)$. 
We will show that a significant proportion of vertices in $L$ possess a degree near $n$. 
Consequently, leveraging Lemma~\ref{LEMMA:2nd-perfect-Ks1s2-bipartite}, we can show that the size of $L$ is less than $k(t+1)$.
Following this, a more detailed analysis of the structure of $G-L$ will show that $G-L$ is $\{K_{1,2}, M_2\}$-free, and hence, contains at most one edge. 
Here, $M_2$ denotes the matching with two edges.

\begin{proof}[Proof of Theorem~\ref{THM:2nd-C2k-main}]
    Fix an integer $k \ge 3$, and let $\mathbb{C} := C_{2k}$ for simplicity. 
    Let 
    \begin{align*}
        \delta_1 := \frac{1}{32k^2}, 
        \quad  
        \delta_2 := \frac{1}{4k}, 
        \quad 
        \delta_3 := \frac{1}{2(k+1)}, \quad 
        \varepsilon := \frac{\delta_1}{144 k^2}, 
        \quad
        \theta:= \frac{20}{\delta_1 k}\frac{\mathrm{ex}(n,\mathbb{C})}{n}. 
    \end{align*}
    Let $n$ be sufficiently large and $t$ be an integer satisfying 
    \begin{align}\label{equ:2nd-C2k-t-range}
        \frac{72 k}{\delta_3} \theta
        \le t 
        \le \varepsilon n.
    \end{align}
    Let $G$ be a maximum $n$-vertex $(t+1)\mathbb{C}$-free graph with vertex set $V$. 
    Let 
    \begin{align*}
        L_{1} := \left\{v\in V \colon d_{G}(v) \ge (1-\delta_1)n\right\},
        \quad  
        L_2  := \left\{v\in V\setminus L_1 \colon d_{G}(v) \ge \delta_2 n\right\}, 
    \end{align*}
    and 
    \begin{align*}
        L & := L_1\cup L_2, \quad  
        U := V \setminus L, \quad 
        \ell:= |L|, \quad 
        \ell_1 := |L_1|, \quad 
        \ell_2:= |L_2|. 
    \end{align*}
    In addition, let 
    \begin{align*}
        G_1:= G - L_1,\quad 
        G_2:= G[U], \quad 
        n_1:= n-\ell_1, 
        \quad\text{and}\quad
         t_1 := t - \left\lfloor \ell_1/k \right\rfloor. 
    \end{align*}
    \begin{claim}\label{CLAIM:2nd-H1-upper-bound}
        We have $\ell_1 < k(t+1)$ and $G_1$ is $(t_1 + 1)\mathbb{C}$-free. 
    \end{claim}
    \begin{proof}
        First, suppose to the contrary that $\ell_1 \ge k(t+1)$. 
        We may assume that $\ell_1 = k(t+1)$ since otherwise we can replace $L_1$ by a subset of size $k(t+1)$. 
        Let $H$ be the collection of edges in $G$ that have exactly one vertex in $L_1$. 
        Notice that $H$ is bipartite, and for every $v\in L_1$ it follows from the assumptions on $t$ and $\delta_1$ that 
        \begin{align*}
            d_{H}(v)
            \ge d_{G}(v) - |L_1| 
             \ge (1-\delta_1)n - k(t+1) 
             \ge (1-\delta_1)n - 2\varepsilon k n 
             > \left(1- 2\delta_1\right)n. 
        \end{align*}
        Since, in addition, $\frac{\left(1-2k\delta_1\right)k}{k}(n-\ell_1) \ge n/4 \ge \ell_1$ (by~\eqref{equ:2nd-C2k-t-range}), 
        it follows from Lemma~\ref{LEMMA:2nd-perfect-Ks1s2-bipartite-a} that $(t+1)\mathbb{C} \subset (t+1)K_{k,k} \subset H$, a contradiction.

        Now suppose to the contrary that there exists a collection  $\mathcal{B} = \{\mathbb{C}_1, \ldots, \mathbb{C}_{t_1+1}\}$ of $t_1+1$ pairwise vertex-disjoint copies of $\mathbb{C}$ in $G_1$. 
        Let $B := \bigcup_{i\in [t_1+1]}V(\mathbb{C}_i)$ and $V' := V\setminus B$.
        Let $H_1$ be the collection of edges in $G[V']$ that contain exactly one vertex in $L_1$. 
        Similar to the argument above, we have 
        \begin{align*}
            d_{H_1}(v)
            \ge d_{G}(v) - |L_1\cup B_1| 
             \ge (1-\delta_1)n - 2k(t+1) 
             > \left(1- 2\delta_1\right)n. 
        \end{align*}
        Similarly, it follows from~\eqref{equ:2nd-C2k-t-range} and  Lemma~\ref{LEMMA:2nd-perfect-Ks1s2-bipartite-a} that  $\lfloor \ell_1/k \rfloor \mathbb{C} \subset H_1$, which together with $\mathcal{B}$ implies that $(t+1)\mathbb{C}\subset G$, a contradiction.  
    \end{proof}
    By Claim~\ref{CLAIM:2nd-H1-upper-bound}, we obtain 
    \begin{align*}
        n_1 
        = n-\ell_1  
        \ge n- k (t+1) 
        \ge (1-2\varepsilon k)n
        \ge \left(1-\frac{\delta_1}{2}\right)n. 
    \end{align*}
    Combined with the definition of $L_1$, we obtain     \begin{align*}
        \Delta(G_1)
        \le (1-\delta_1)n
         = (1-\delta_1) \times \frac{n}{n_1} \times n_1 
         \le \left(1-\frac{\delta_1}{2}\right)n_1. 
    \end{align*}
    Notice from the definition of $\theta$ and Proposition~\ref{PROP:C2k-free-lower-bound} that  
    \begin{align*}
        \theta
        = \frac{20}{\delta_1 k}\frac{\mathrm{ex}(n,\mathbb{C})}{n}
        \ge \frac{5}{\delta_1 k/2}\left(\frac{\mathrm{ex}(n_1,\mathbb{C})}{n_1-1} + \frac{72k}{\delta_1/2}\right). 
    \end{align*}
    If $t_1 \ge  \theta$, then 
    it follows from $t_1 \le t \le \varepsilon n \le \frac{\delta_1 n_1/2}{72k^2}$ and Lemma~\ref{LEMMA:2nd-weak-bound-Ks1s2} that 
    \begin{align*}
        |G_1|
        \le \binom{n_1}{2} - \binom{n_1-k(t_1-1)}{2}
        = \binom{n-\ell_1}{2} - \binom{n-k t}{2}. 
    \end{align*}
    Consequently, 
    \begin{align*}
        |G|
        \le \binom{n}{2} -\binom{n-\ell_1}{2} + |G_1|
        \le \binom{n}{2} - \binom{n-k t}{2},  
    \end{align*}
    and we are done. 
    
    So we may assume that $t_1 \le \theta$. 
    It follows from $t \ge 80k\theta/\delta_3$ and the definition of $t_1$ that 
    \begin{align}\label{equ:2nd-hygp-L1-lower-bound}
        \ell_1 
        \ge k \left(t - t_1\right)
        \ge k \left(t - \theta \right)
        \ge k \left(t - \frac{t-k+1}{k}\right)
        =  \frac{k-1}{k} \times k(t+1). 
    \end{align}
    Let 
    \begin{align*}
        t_2
        := t- \left\lfloor \ell/k \right\rfloor
        \le t_1
        \le \theta. 
    \end{align*}
    \begin{claim}\label{CLAIM:2nd-L1+L2-upper-bound}
        We have $\ell_1 + \ell_2 \le k(t+1)-1$, and $G[U]$ is $(t_2+1) \mathbb{C}$-free. 
        In particular, 
        \begin{align}\label{equ:THM-2nd-pg-GU-upper-bound}
                    |G[U]|
                    \le 2k t_2 \times \delta_2 n + \mathrm{ex}(n-\ell-2kt_2, \mathbb{C})
                    \le \frac{\theta n}{2} +  \mathrm{ex}(n, \mathbb{C})
                    \le \theta n. 
        \end{align}  
    \end{claim}
    \begin{proof}
        The proof is similar to that of Claim~\ref{CLAIM:2nd-H1-upper-bound}. 
        Suppose to the contrary that Claim~\ref{CLAIM:2nd-L1+L2-upper-bound} is not true. 
        If $\ell < k(t+1)$, then let $\mathcal{B}' := \left\{\mathbb{C}_1', \ldots, \mathbb{C}_{t_2+1}' \right\}$ be a collection of pairwise vertex-disjoint copies of $\mathbb{C}$ in $G[U]$ and let $B':= \bigcup_{i\in [t_2+1]}V(\mathbb{C}_i')$. 
        If $\ell \ge k(t+1)$, then by removing vertices in $L_2$, we may assume that $\ell= k(t+1)$. In addition, we set $B':= \emptyset$ in this case. 
         Let $H_2$ be the collection of all edges in $G[V\setminus B']$ that contain exactly one vertex in $L$.
        To build a contradiction, it suffices to show that $\lfloor \ell/k \rfloor \mathbb{C} \subset H_2$.  
        Indeed, observe that for every $v\in L_1$, we have 
        \begin{align*}
            d_{H_2}(v)
             \ge d_{G}(v) - |L\cup B'|   
             \ge (1-\delta_1)n - 2k(t+1)
             \ge (1-\delta_1)n - 4\varepsilon kn
             \ge \left(1 - \frac{\delta_2}{4k}\right)n, 
        \end{align*}
        and for every $u\in L_2$, we have 
        \begin{align*}
            d_{H_2}(u)
             \ge d_{G}(u) - |L\cup B'|   
             \ge \delta_2n - 4\varepsilon kn
             \ge \frac{\delta_2 n}{2}.  
        \end{align*}
        Since, in addition, $\ell \le 2kt \le 2\varepsilon k n \le  \frac{\delta_2 n/4}{8k} \le \frac{\delta_2 (n-|B'|)/2}{8k}$, 
         it follows from~\eqref{equ:2nd-hygp-L1-lower-bound} and Lemma~\ref{LEMMA:2nd-perfect-Ks1s2-bipartite} that $\lfloor \ell/k \rfloor \mathbb{C} \subset H_2$, a contradiction. 
    \end{proof}
    We divide $U$ further by letting 
    \begin{align*}
        S
        := \left\{u\in U \colon |N_{G}(u) \cap L| \le (1-\delta_3)|L|\right\} 
        \quad\text{and}\quad 
        W := U \setminus S. 
    \end{align*}
    In particular, since, by Claim~\ref{CLAIM:2nd-L1+L2-upper-bound}, $G[S] \subset G[U]$ is $(t_2+1) \mathbb{C}$-free, it follows from Lemma~\ref{LEMMA:trivial-max-degree} that 
    \begin{align}\label{equ:2nd-graph-G[S]-upper-bound}
        |G[S]|
        \le 2k t_2 |S| + \mathrm{ex}\left(|S|, \mathbb{C}\right)
        \le 2k \theta |S| + \mathrm{ex}\left(|S|, \mathbb{C}\right). 
    \end{align}

    It follows from the definition  of $S$ that 
    \begin{align}\label{equ:THM-2nd-pg-G[L,U]-upper-bound}
        |G[L,U]|
        \le |W||L| + |S|\times (1-\delta_3)|L|
        = \ell (n-\ell) - \delta_3 \ell |S|. 
    \end{align}
    Combining~\eqref{equ:THM-2nd-pg-GU-upper-bound} and~\eqref{equ:THM-2nd-pg-G[L,U]-upper-bound}, we obtain 
        \begin{align*}
            |G|
              = |G[L]| + |G[L,U]| + |G[U]| 
            & \le \binom{\ell}{2} + \ell (n-\ell) - \delta_3 \ell |S|
            + \theta n. 
        \end{align*}
    So it follows from  $|G|\ge |S(k(t+1)-1,n)| = \binom{n}{2}-\binom{n-k(t+1)+1}{2} \ge \binom{\ell}{2} + \ell (n-\ell)$ that  
        \begin{align*}
            |S|
            \le \frac{\theta n}{k(t-\theta)}
            < \frac{n}{2}. 
        \end{align*}
    which combined with Proposition~\ref{PROP:Turan-ratio} implies that  
    \begin{align*}
        \frac{\mathrm{ex}\left(n,\mathbb{C}\right)}{n}
        \ge \left(1 - \frac{|S|}{n}\right)\frac{\mathrm{ex}\left(|S|,\mathbb{C}\right)}{|S|}
        \ge \frac{1}{2}\frac{\mathrm{ex}\left(|S|,\mathbb{C}\right)}{|S|}.
    \end{align*}
    Consequently, by~\eqref{equ:2nd-C2k-t-range}, we obtain 
    \begin{align}\label{equ:2nd-graph-ratio}
        \frac{\delta_3 \ell}{2}|S| - \mathrm{ex}\left(|S|, \mathbb{C}\right)
          \ge \left(\frac{\delta_3}{2} \times k(t-\theta) - 2\frac{\mathrm{ex}\left(n, \mathbb{C}\right)}{n} \right)|S|
         \ge 0. 
    \end{align}
    In addition, 
    \begin{align}\label{equ:2nd-graph-tech-b}
        \frac{\delta_3 \ell}{2} - 2k \theta - 16k^2 \theta
        \ge \frac{\delta_3}{2}\times k(t-\theta) - 2k \theta - 16k^2 \theta
        > 0. 
    \end{align}
    Let 
    \begin{align*}
        t_3:= k(t+1)-1-\ell, \quad 
        S_1 := \left\{v\in S \colon |N_{G}(v) \cap W| \ge 16k^2 \theta \right\}, \quad\text{and}\quad
        x := |S_1|. 
    \end{align*}
    It follows from~\eqref{equ:2nd-graph-G[S]-upper-bound},~\eqref{equ:2nd-graph-ratio},~\eqref{equ:2nd-graph-tech-b} and definition of $S_1$ that 
    \begin{align*}
        |G[S]| + |G[L,S]|+ |G[S,W]| 
         & \le  2k \theta |S| + \mathrm{ex}\left(|S|, \mathbb{C}\right)
            + (1-\delta_3)\ell |S| + x \Delta + 16k^2 \theta |S| \notag \\
         & \le  \ell |S| - \left(\delta_3 \ell - 2k \theta - 16k^2 \theta\right) |S| + \mathrm{ex}\left(|S|, \mathbb{C}\right) + \frac{xn}{2} \notag \\
         & \le \ell |S| + \frac{xn}{2} - \frac{\delta_3 \ell}{2}|S| + \mathrm{ex}\left(|S|, \mathbb{C}\right) 
          \le \ell |S| + \frac{xn}{2}. 
    \end{align*}
    Consequently, 
    \begin{align}\label{equ:2nd-graph-G[L,U]-G[S]-G[S,W]}
        |G[L,U]| + |G[S]| + |G[S,W]| 
        & = |G[L,W]|+ |G[S]|+|G[L,S]|+|G[L,W]| \notag \\
        & \le \ell |W| + \ell |S| + \frac{xn}{2}
        = \ell(n-\ell) + \frac{xn}{2}. 
    \end{align}
    
    For every set $T \subset U$ let $N_T:= \bigcap_{u\in T}N_{G}(u) \cap L$ denote the set of common neighbors of $T$ in $L$. 
    It follows from the definition of $W$ that for every $(k+1)$-set $T\subset U$ we have 
    \begin{align}\label{equ:THM-2nd-pg-NT-lower-bound}
        |N_T| 
        \ge \ell - \sum_{u\in T}\left(\ell - |N_{G}(u) \cap L|\right)
        \ge \left(1-\delta_3 (k+1)\right)\ell
        = \frac{\ell}{2}. 
    \end{align} 
    
\begin{figure}[htbp]
\centering
\tikzset{every picture/.style={line width=0.85pt}} 

\begin{tikzpicture}[x=0.75pt,y=0.75pt,yscale=-1,xscale=1]

\draw [line width=1pt]  (527.22,41.47) .. controls (531.34,41.47) and (534.64,44.81) .. (534.6,48.92) -- (534.41,71.26) .. controls (534.37,75.37) and (531,78.7) .. (526.89,78.7) -- (165.91,78.45) .. controls (161.8,78.45) and (158.49,75.11) .. (158.53,71) -- (158.73,48.66) .. controls (158.77,44.54) and (162.13,41.21) .. (166.24,41.21) -- cycle ;
\draw [line width=1pt]  (515.35,93.2) .. controls (526.94,93.21) and (536.25,102.61) .. (536.15,114.2) -- (535.59,177.15) .. controls (535.48,188.73) and (526,198.12) .. (514.42,198.11) -- (182.14,197.88) .. controls (170.55,197.87) and (161.24,188.47) .. (161.35,176.89) -- (161.91,113.94) .. controls (162.01,102.35) and (171.49,92.96) .. (183.08,92.97) -- cycle ;
%
\draw    (535,177) -- (161,177) ;
\draw    (440,177) -- (440,198) ;

\draw    (180,60) -- (180,130) ;
\draw    (200,60) -- (200,130) ;
\draw    (220,60) -- (220,130) ;
\draw    (180,60) -- (220,130) ;
\draw    (200,60) -- (180,130) ;
\draw    (220,60) -- (200,130) ;
\draw [fill=uuuuuu]   (180,60) circle (1.5pt);
\draw [fill=uuuuuu]   (180,130)  circle (1.5pt);
\draw [fill=uuuuuu]   (200,60) circle (1.5pt);
\draw [fill=uuuuuu]   (200,130)  circle (1.5pt);
\draw [fill=uuuuuu]   (220,60)  circle (1.5pt);
\draw [fill=uuuuuu]   (220,130)  circle (1.5pt);
%
\draw    (250,60) -- (250,130) ;
\draw    (270,60) -- (250,130) ;
\draw    (270,60) -- (270,130) ;
\draw    (290,60) -- (270,130) ;
\draw    (290,60) -- (290,130) ;
\draw    (250,60) -- (290,130) ;
\draw [fill=uuuuuu]   (250,60) circle (1.5pt);
\draw [fill=uuuuuu]   (250,130)  circle (1.5pt);
\draw [fill=uuuuuu]   (270,60) circle (1.5pt);
\draw [fill=uuuuuu]   (270,130)  circle (1.5pt);
\draw [fill=uuuuuu]   (290,60)  circle (1.5pt);
\draw [fill=uuuuuu]   (290,130)  circle (1.5pt);
%
\draw    (315,130) -- (340,130) ;
\draw    (355,130) -- (380,130) ;
\draw    (315,130) -- (335,60) ;
\draw    (335,60) -- (355,130) ;
\draw    (340,130) -- (360,60) ;
\draw    (360,60) -- (380,130) ;
\draw [fill=uuuuuu]   (315,130) circle (1.5pt);
\draw [fill=uuuuuu]   (340,130)  circle (1.5pt);
\draw [fill=uuuuuu]   (355,130) circle (1.5pt);
\draw [fill=uuuuuu]   (380,130)  circle (1.5pt);
\draw [fill=uuuuuu]   (335,60)  circle (1.5pt);
\draw [fill=uuuuuu]   (360,60)  circle (1.5pt);
%
\draw   (420,160)  -- (400,130) ;
\draw    (420,160) -- (440,130) ;
\draw    (410,60) -- (400,130) ;
\draw    (410,60) -- (420,130) ;
\draw    (430,60) -- (420,130) ;
\draw    (430,60) -- (440,130) ;
\draw [fill=uuuuuu]   (420,160) circle (1.5pt);
\draw [fill=uuuuuu]   (400,130)  circle (1.5pt);
\draw [fill=uuuuuu]   (440,130) circle (1.5pt);
\draw [fill=uuuuuu]   (420,130)  circle (1.5pt);
\draw [fill=uuuuuu]   (410,60)  circle (1.5pt);
\draw [fill=uuuuuu]   (430,60)  circle (1.5pt);
%
\draw    (490,190) -- (470,130) ;
\draw    (490,190) -- (510,130) ;
\draw    (480,60) -- (490,130) ;
\draw    (480,60) -- (470,130) ;
\draw    (500,60) -- (490,130) ;
\draw    (500,60) -- (510,130) ;
\draw [fill=uuuuuu]   (490,190) circle (1.5pt);
\draw [fill=uuuuuu]   (470,130)  circle (1.5pt);
\draw [fill=uuuuuu]   (510,130) circle (1.5pt);
\draw [fill=uuuuuu]   (490,130)  circle (1.5pt);
\draw [fill=uuuuuu]   (500,60)  circle (1.5pt);
\draw [fill=uuuuuu]   (480,60)  circle (1.5pt);
%
\draw (135,55) node [anchor=north west][inner sep=0.75pt]   [align=left] {$L$};
\draw (135,130) node [anchor=north west][inner sep=0.75pt]   [align=left] {$W$};
\draw (135,180) node [anchor=north west][inner sep=0.75pt]   [align=left] {$S$};
\draw (478,205) node [anchor=north west][inner sep=0.75pt]   [align=left] {$S_1$};
\end{tikzpicture}
\caption{Supplementary diagram for Claim~\ref{CLAIM:2nd-graph-S1-G[W]} when $k=3$.}
\label{fig:2nd-graph}
\end{figure}
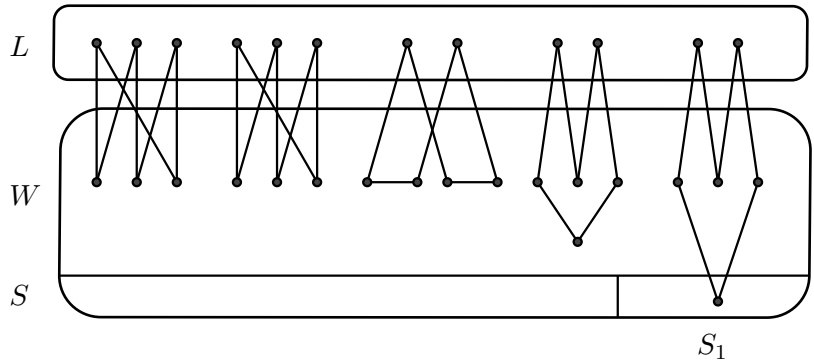

    Let $\mathbb{C}':= \mathbb{C}[k+1]$ and note that every $\mathbb{C}'$-free graph with at least $k+1$ vertices is also $\{K_{1,2},M_2\}$-free (see Figure~\ref{fig:C2k-[k+1]}). 
    \begin{claim}\label{CLAIM:2nd-graph-S1-G[W]}
        We have $x \le t_3$ and $G[W]$ is $(t_3-x+1)\cdot \{K_{1,2}, M_{2}\}$-free. 
        Moreover,  
        \begin{align}\label{equ:2nd-graph-G[W]-upper-bound-1}
            |G[W]|
            \le \frac{t_3-x}{2}n + 1. 
        \end{align}
    \end{claim}
    \begin{proof}
        First, suppose to the contrary that there exists a $(t_3+1)$-set $\{v_1, \ldots, v_{t_3+1}\} \subset S_1$ such that $|N_{G}(v_i) \cap W| \ge 16k^2 \theta \ge  4k (t_3+1)$ for all $i\in [t_3+1]$. 
        Choose an $k$-set $X_i$ from $N_{G}(v_i) \cap W$ for each $i\in [t_2+1]$ such that $\{X_1, \ldots, X_{t_3+1}\}$ are pairwise disjoint (the definition of $S$ ensures that this is possible). 
        For each $N_i$ choose an $(k-1)$-set $Y_i \subset N_{X_i}$ such that $\{Y_1, \ldots, Y_{t_3+1}\}$ are pairwise disjoint (the existence of such sets are guaranteed by~\eqref{equ:THM-2nd-pg-NT-lower-bound}). 
        Notice that $\mathbb{C} \subset K_{k, k} \subset G[\{v_i\} \cup X_i \cup Y_i]$ for $i\in [t_3+1]$. 
        
        Let $Z:= \bigcup_{i\in [t_3+1]}\left(\{v_i\} \cup X_i \cup Y_i\right)$ and $G'$ be the induced subgraph of $G$ on $V\setminus Z$.
        For every $v\in L_1 \setminus Z$ we have 
        \begin{align*}
            d_{G'}(v) 
            \ge d_{G}(v) - |Z|
            \ge (1-\delta_1)n - 2k(t_3+1)
            \ge (1-2\delta_1)n, 
        \end{align*}
        and for every $u\in L_2 \setminus Z$ we have 
        \begin{align*}
            d_{G'}(v) 
            \ge d_{G}(v) - |Z|
            \ge \delta_2 n - 2k(t_3+1)
            \ge \frac{\delta_2}{2}n.  
        \end{align*}
        Since, in addition, $\ell \le k(t+1) \le 2\varepsilon k n \le \frac{\delta_2 (n-\ell-|Z|)/2}{8k}$, 
        it follows from Lemma~\ref{LEMMA:2nd-perfect-Ks1s2-bipartite} that $G'$ contains at least 
        \begin{align*}
            \left\lfloor \frac{\ell- (k-1)t_3}{k} \right\rfloor
            = \ell - (k-1)(t+1)
        \end{align*}
        pairwise vertex-disjoint copies of $\mathbb{C}$. 
        This implies that $G$ contains 
        \begin{align*}
            t_3+1 + \ell - (k-1)(t+1) 
            = k(t+1)-1-\ell + 1 + \ell - (k-1)(t+1)
            = t+1
        \end{align*}
        pairwise vertex-disjoint copies of $\mathbb{C}$, a contradiction. 
        This proves that $x \le t_3$. 
        The proof for the $(t_3-x+1) \cdot \{K_{1,2}, M_{2}\}$-freeness of $G[W]$ is similar (see Figures~\ref{fig:C2k-[k+1]} and~\ref{fig:2nd-graph}), so we omit the details. 

        Recall that $\Delta(G[W]) \le \Delta(G[U]) \le \delta_3 n$, so it follows from Lemma~\ref{LEMMA:trivial-max-degree} and $(t_3-x+1)\cdot \{K_{1,2}, M_{2}\}$-freeness of $G[W]$ that 
        \begin{align*}
            |G[W]|
            \le (t_3-x)(k+1)\times \delta_3 n + \mathrm{ex}(n,\{K_{1,2},M_2\})
            \le \frac{t_3-x}{2}n+ 1,
        \end{align*}
        completing the proof of Claim~\ref{CLAIM:2nd-graph-S1-G[W]}.
    \end{proof}
    It follows from~\eqref{equ:2nd-graph-G[L,U]-G[S]-G[S,W]},~\eqref{equ:2nd-graph-G[W]-upper-bound-1}, and Fact~\ref{FACT:binom-inequality-b} that 
    \begin{align*}
        |G|
        & = |G[L]| + |G[L,U]| + |G[S]| + |G[S,W]| + |G[W]| \notag \\
        & \le \binom{\ell}{2}  + \ell(n-\ell) + \frac{xn}{2} + \frac{t_3-x}{2}n + 1 \notag \\
        & \le \binom{n}{2}-\binom{n-k(t+1)+1}{2} - t_3 \times \frac{2n}{3}+ \frac{t_3}{2}n + 1 \notag \\
        & \le \binom{n}{2}-\binom{n-k(t+1)+1}{2} + 1, 
    \end{align*}
    completing the proof of Theorem~\ref{THM:2nd-C2k-main}. 
\end{proof}
\bibliographystyle{abbrv}

\begin{thebibliography}{10}

\bibitem{ABHP15}
P.~Allen, J.~B\"{o}ttcher, J.~Hladk\'{y}, and D.~Piguet.
\newblock A density {C}orr\'{a}di-{H}ajnal theorem.
\newblock {\em Canad. J. Math.}, 67(4):721--758, 2015.

\bibitem{Ben66}
C.~T. Benson.
\newblock Minimal regular graphs of girths eight and twelve.
\newblock {\em Canadian J. Math.}, 18:1091--1094, 1966.

\bibitem{BS74}
J.~A. Bondy and M.~Simonovits.
\newblock Cycles of even length in graphs.
\newblock {\em J. Combinatorial Theory Ser. B}, 16:97--105, 1974.

\bibitem{Brown66}
W.~G. Brown.
\newblock On graphs that do not contain a {T}homsen graph.
\newblock {\em Canad. Math. Bull.}, 9:281--285, 1966.

\bibitem{BJ17}
B.~Bukh and Z.~Jiang.
\newblock A bound on the number of edges in graphs without an even cycle.
\newblock {\em Combin. Probab. Comput.}, 26(1):1--15, 2017.

\bibitem{BJ17b}
B.~Bukh and Z.~Jiang.
\newblock Erratum for `{A} bound on the number of edges in graphs without an
  even cycle'.
\newblock {\em Combin. Probab. Comput.}, 26(6):952--953, 2017.

\bibitem{CFKS14}
S.~Chiba, S.~Fujita, K.-i. Kawarabayashi, and T.~Sakuma.
\newblock Minimum degree conditions for vertex-disjoint even cycles in large
  graphs.
\newblock {\em Adv. in Appl. Math.}, 54:105--120, 2014.

\bibitem{CH63}
K.~Corradi and A.~Hajnal.
\newblock On the maximal number of independent circuits in a graph.
\newblock {\em Acta Math. Acad. Sci. Hungar.}, 14:423--439, 1963.

\bibitem{Ega96}
Y.~Egawa.
\newblock Vertex-disjoint cycles of the same length.
\newblock {\em J. Combin. Theory Ser. B}, 66(2):168--200, 1996.

\bibitem{Erdos62}
P.~Erd\H{o}s.
\newblock \"{U}ber ein {E}xtremalproblem in der {G}raphentheorie.
\newblock {\em Arch. Math. (Basel)}, 13:222--227, 1962.

\bibitem{Erd64}
P.~Erd{\H o}s.
\newblock Extremal problems in graph theory.
\newblock In {\em Theory of {G}raphs and its {A}pplications ({P}roc. {S}ympos.
  {S}molenice, 1963)}, pages 29--36. Publ. House Czech. Acad. Sci., Prague,
  1964.

\bibitem{ER62}
P.~Erd{\H o}s and A.~R\'{e}nyi.
\newblock On a problem in the theory of graphs.
\newblock {\em Magyar Tud. Akad. Mat. Kutat\'{o} Int. K\"{o}zl.}, 7:623--641,
  1962.

\bibitem{ERS66}
P.~Erd{\H o}s, A.~R\'{e}nyi, and V.~T. S\'{o}s.
\newblock On a problem of graph theory.
\newblock {\em Studia Sci. Math. Hungar.}, 1:215--235, 1966.

\bibitem{Fur83}
Z.~F\"{u}redi.
\newblock Graphs without quadrilaterals.
\newblock {\em J. Combin. Theory Ser. B}, 34(2):187--190, 1983.

\bibitem{Fur94}
Z.~F{\"u}redi.
\newblock Quadrilateral-free graphs with maximum number of edges.
\newblock {\em preprint}, 1994.

\bibitem{Fur96}
Z.~F\"{u}redi.
\newblock On the number of edges of quadrilateral-free graphs.
\newblock {\em J. Combin. Theory Ser. B}, 68(1):1--6, 1996.

\bibitem{FNV06}
Z.~F\"{u}redi, A.~Naor, and J.~Verstra\"{e}te.
\newblock On the {T}ur\'{a}n number for the hexagon.
\newblock {\em Adv. Math.}, 203(2):476--496, 2006.

\bibitem{HW15}
D.~J. Harvey and D.~R. Wood.
\newblock Cycles of given size in a dense graph.
\newblock {\em SIAM J. Discrete Math.}, 29(4):2336--2349, 2015.

\bibitem{He21}
Z.~He.
\newblock A new upper bound on the {T}u\'{r}an number of even cycles.
\newblock {\em Electron. J. Combin.}, 28(2):Paper No. 2.41, 18, 2021.

\bibitem{HHLLYZ23a}
J.~Hou, C.~Hu, H.~Li, X.~Liu, C.~Yang, and Y.~Zhang.
\newblock Toward a density {C}orr{\' a}di--{H}ajnal theorem for degenerate
  hypergraphs.
\newblock preprint.

\bibitem{HLLYZ23}
J.~Hou, H.~Li, X.~Liu, L.-T. Yuan, and Y.~Zhang.
\newblock A step towards a general density {C}orr{\' a}di--{H}ajnal theorem.
\newblock {\em arXiv preprint arXiv:2302.09849}, 2023.

\bibitem{Jus85}
P.~Justesen.
\newblock On independent circuits in finite graphs and a conjecture of {E}rd{\H
  o}s and {P}\'{o}sa.
\newblock In {\em Graph theory in memory of {G}. {A}. {D}irac ({S}andbjerg,
  1985)}, volume~41 of {\em Ann. Discrete Math.}, pages 299--305.
  North-Holland, Amsterdam, 1989.

\bibitem{LU95}
F.~Lazebnik and V.~A. Ustimenko.
\newblock Explicit construction of graphs with an arbitrary large girth and of
  large size.
\newblock volume~60, pages 275--284. 1995.
\newblock ARIDAM VI and VII (New Brunswick, NJ, 1991/1992).

\bibitem{LUW94}
F.~Lazebnik, V.~A. Ustimenko, and A.~J. Woldar.
\newblock Properties of certain families of {$2k$}-cycle-free graphs.
\newblock {\em J. Combin. Theory Ser. B}, 60(2):293--298, 1994.

\bibitem{LUW95}
F.~Lazebnik, V.~A. Ustimenko, and A.~J. Woldar.
\newblock A new series of dense graphs of high girth.
\newblock {\em Bull. Amer. Math. Soc. (N.S.)}, 32(1):73--79, 1995.

\bibitem{LUW99}
F.~Lazebnik, V.~A. Ustimenko, and A.~J. Woldar.
\newblock Polarities and {$2k$}-cycle-free graphs.
\newblock volume 197/198, pages 503--513. 1999.
\newblock 16th British Combinatorial Conference (London, 1997).

\bibitem{Mantel07}
W.~Mantel.
\newblock Vraagstuk {XXVIII}.
\newblock {\em Wiskundige Opgaven}, 10(2):60--61, 1907.

\bibitem{Mel04}
K.~E. Mellinger.
\newblock L{DPC} codes from triangle-free line sets.
\newblock {\em Des. Codes Cryptogr.}, 32(1-3):341--350, 2004.

\bibitem{MM05}
K.~E. Mellinger and D.~Mubayi.
\newblock Constructions of bipartite graphs from finite geometries.
\newblock {\em J. Graph Theory}, 49(1):1--10, 2005.

\bibitem{Moon68}
J.~W. Moon.
\newblock On independent complete subgraphs in a graph.
\newblock {\em Canadian J. Math.}, 20:95--102, 1968.

\bibitem{Oleg12}
O.~Pikhurko.
\newblock A note on the {T}ur{\'a}n function of even cycles.
\newblock {\em Proceedings of the American Mathematical Society},
  140(11):3687--3692, 2012.

\bibitem{Sim66}
M.~Simonovits.
\newblock A method for solving extremal problems in graph theory, stability
  problems.
\newblock In {\em Theory of {G}raphs ({P}roc. {C}olloq., {T}ihany, 1966)},
  pages 279--319. Academic Press, New York-London, 1968.

\bibitem{SI68}
M.~Simonovits.
\newblock A method for solving extremal problems in graph theory, stability
  problems.
\newblock In {\em Theory of {G}raphs ({P}roc. {C}olloq., {T}ihany, 1966)},
  pages 279--319. Academic Press, New York, 1968.

\bibitem{Tho83}
C.~Thomassen.
\newblock Girth in graphs.
\newblock {\em J. Combin. Theory Ser. B}, 35(2):129--141, 1983.

\bibitem{Ver00}
J.~Verstra\"{e}te.
\newblock On arithmetic progressions of cycle lengths in graphs.
\newblock {\em Combin. Probab. Comput.}, 9(4):369--373, 2000.

\bibitem{Ver03}
J.~Verstra\"{e}te.
\newblock Vertex-disjoint cycles of the same length.
\newblock {\em J. Combin. Theory Ser. B}, 88(1):45--52, 2003.

\bibitem{Wang12}
H.~Wang.
\newblock An extension of the {C}orr\'{a}di-{H}ajnal theorem.
\newblock {\em Australas. J. Combin.}, 54:59--84, 2012.

\bibitem{Wen91}
R.~Wenger.
\newblock Extremal graphs with no {$C^4$}'s, {$C^6$}'s, or {$C^{10}$}'s.
\newblock {\em J. Combin. Theory Ser. B}, 52(1):113--116, 1991.

\end{thebibliography}

\end{document}